\documentclass[12pt]{amsart}

\usepackage[english]{babel}
\usepackage{microtype}

\usepackage{amsmath,amssymb,amsfonts,amsthm,mathtools}
\usepackage{booktabs}

\usepackage{xcolor}
\usepackage{todonotes}
\usepackage[colorlinks,linkcolor=blue,anchorcolor=blue,citecolor=blue,backref=page]{hyperref}
\hypersetup{breaklinks=true,bookmarksdepth=2}

\usepackage[norefs,nocites]{refcheck}

\allowdisplaybreaks

\theoremstyle{plain}
\newtheorem{theorem}{Theorem}[section]
\newtheorem{proposition}[theorem]{Proposition}
\newtheorem{lemma}[theorem]{Lemma}

\numberwithin{equation}{section}

\DeclareMathOperator{\Kl}{Kl}
\newcommand{\e}{\mathrm e}
\newcommand{\eps}{\varepsilon}
\newcommand{\om}{\omega}
\newcommand{\one}{\mathbf 1}
\newcommand{\calL}{\mathcal L}
\newcommand{\calR}{\mathcal R}
\newcommand{\calT}{\mathcal T}
\newcommand{\abs}[1]{\left|#1\right|}
\newcommand{\norm}[1]{\left\lVert#1\right\rVert}

\title[Selberg sieve weights and Kloosterman sign changes, II]
{Selberg sieve weights and sign changes of Kloosterman sums. II.
Square-free moduli with at most four prime factors}

\author{Yixiu Xiao}
\address{School of Mathematical Sciences, Shanghai Jiao Tong University,
800 Dongchuan Road, Shanghai 200240, China}
\email{yixiuxiao98@gmail.com}

\author{Hongze Li}
\address{School of Mathematical Sciences, Shanghai Jiao Tong University,
800 Dongchuan Road, Shanghai 200240, China}
\email{lihz@sjtu.edu.cn}
\thanks{Hongze Li is the corresponding author.}

\subjclass[2020]{Primary 11L05; Secondary 11N36}
\keywords{Kloosterman sums, sign changes, square-free almost primes,
Selberg sieve, mean-square estimates}

\begin{document}

\begin{abstract}
Let \(\Kl(1;q)\) denote the normalized Kloosterman sum modulo \(q\).
We prove that, for each sign, there are \(\gg X/\log X\) square-free
moduli \(q\in(X,2X]\) with at most four prime factors for which
\(\Kl(1;q)\) has that sign.  The proof combines the Selberg sieve with
equidistribution and mean-square estimates for Kloosterman sums; the
uniform estimate for shifted sieve weights established in Part~I
controls the contribution of moduli having small prime factors.
\end{abstract}

\maketitle
\setcounter{tocdepth}{2}
\tableofcontents

\section{Introduction}
\subsection{Statement of the result}

For a positive integer \(q\) and an integer \(m\) coprime to \(q\),
define the normalized Kloosterman sum
\[
  \Kl(m;q)
  =\frac{1}{\sqrt q}
    \sum_{\substack{x\bmod q\\(x,q)=1}}
      \e\!\left(\frac{mx+\overline{x}}{q}\right),
  \qquad \e(z)=\e^{2\pi iz},
\]
where \(\overline{x}\) is the multiplicative inverse of \(x\) modulo
\(q\).  The sum is real, and the Weil bound gives
\(\abs{\Kl(1;q)}\leq 2^{\om(q)}\) for square-free \(q\), where
\(\om(q)\) denotes the number of distinct prime factors of \(q\).
For brevity, we call a square-free integer \(q\) with
\(\om(q)\leq k\) a \(P_k\)-modulus.

We study the signs of \(\Kl(1;q)\) as the modulus varies.
Even the existence of infinitely many sign changes over prime moduli
remains open. Fouvry and Michel~\cite{FM} established quantitative
sign changes over $P_{23}$ moduli.
This bound was subsequently reduced to $P_{18}$ by
Sivak-Fischler~\cite{Sivak}, to $P_{15}$ by
Matom\"aki~\cite{Matomaki}, to $P_{10}$ by Xi~\cite{Xi1}, and then
to $P_{7}$ in his subsequent work~\cite{Xi2,XiCorr}.
More recently, Zhang and Zhong~\cite{ZZv2} obtained an unconditional
result for \(P_6\)-moduli, and the first author~\cite{XiaoP5}
reduced this to \(P_5\)-moduli. We reduce the number of permitted
prime factors further to $P_{4}$, retaining the quantitative lower
bound \(\gg X/\log X\) for the number of moduli of each sign in
every sufficiently large dyadic interval \((X,2X]\).

\begin{theorem}\label{thm:main}
For each \(\sigma\in\{+1,-1\}\), there is a constant
\(c_\sigma>0\) such that, for all sufficiently large \(X\),
\[
  \#\left\{
    X<q\leq 2X:
    \mu^2(q)=1,\ \om(q)\leq4,\ 
    \sigma\Kl(1;q)>0
  \right\}
  \geq c_\sigma\frac{X}{\log X}.
\]
\end{theorem}

This is the second of two papers on Selberg sieve weights and
Kloosterman sign changes.  Part~I proves the uniform estimate for
shifted sieve weights used below; the statement needed here is recalled
in Lemma~\ref{lem:uniform-selberg}.

\subsection{Main ideas of the proof}

We use a Selberg weight \(W(n)=(\sum_{d\mid n}\lambda_d)^2\) with a
quadratic cutoff profile.  Let \(H_{\leq4}\) and \(H_{\geq5}\) denote
the corresponding absolute Kloosterman mass from integers with at most
four and at least five prime factors, respectively, and let \(S\) be
the signed weighted sum.  If \(M_\sigma\) is the absolute mass carried
by \(P_4\)-moduli of sign \(\sigma\), then the elementary separation in
Lemma~\ref{lem:mass} gives
\begin{equation}
\label{eq:intro-mass-separation}
  M_\sigma
  \geq
  \frac12\left(H_{\leq4}-H_{\geq5}-\abs{S}\right)
  \qquad (\sigma\in\{+1,-1\}).
\end{equation}
By~\eqref{eq:intro-mass-separation}, it is enough to show that \(H_{\leq4}\) exceeds
\(H_{\geq5}\) by a positive multiple of \(X/\log X\), while \(\abs{S}\) is
smaller than this gap.

The lower bound uses only moduli with two or three prime factors.
Equidistribution over the prime factors, combined with
Matom\"aki's rearrangement inequality~\cite{Matomaki}, supplies their
absolute Kloosterman mass.  We use the method of
Fouvry--Michel~\cite{FM} in the form refined by
Xi~\cite{Xi1,Xi2,XiCorr}.

For the complementary mass, reserve the largest prime factor of each
modulus and divide the others alternately into two factors \(d\) and
\(e\).  Twisted multiplicativity and the Weil bound at the reserved
prime give
\[
  |\Kl(1;n)|\leq
  |\Kl(\overline{n/d}^{\,2};d)|^2+
  |\Kl(\overline{n/e}^{\,2};e)|^2.
\]
Both squared traces lie in the range of Xi's mean-square
estimate~\cite[Lemma~4.1]{XiBDH}, apart from a negligible region in
the even prime-factor layers.  On short prime boxes, the sieve weight
may be replaced by a constant with a negligible total error.
Each squared trace then has mean \(1+o(1)\), yielding
\[
  H_{\geq5}(X)\leq
  2\sum_{\om(n)\geq5}g(n/X)\mu^2(n)W(n)
  +o(X/\log X).
\]
A one-dimensional Selberg estimate bounds the total sieve mass.
Subtracting the layers with at most four prime factors gives the
required upper bound.  The resulting mass gap exceeds
\(0.0957X/\log X\), up to the small-prime and asymptotic errors.

It remains to control \(S\).  A Bombieri--Vinogradov-type estimate for
Kloosterman sums applies after the restriction excluding small prime
factors has been removed.  Extracting such a prime shifts the cutoff in
the Selberg weight by \(\log p/\log R\), which may tend to zero.  The
uniform shifted-weight estimate proved in Part~I, recalled in
Lemma~\ref{lem:uniform-selberg}, bounds this correction by
\(O(\eps X/\log X)\) for fixed \(\eps>0\), where \(X^\eps\) is
the pre-sieving threshold.

\subsection{Organization of the paper}
Section~\ref{sec:setup} introduces the sieve, proves the
mass-separation reduction, and establishes signed cancellation.  The
lower-layer estimate is established in Section~\ref{sec:low}.
Sections~\ref{sec:pair} and~\ref{sec:transfer} establish the
second-moment estimate for the contribution from moduli with at least
five prime factors.  Section~\ref{sec:numerics} evaluates the remaining
sieve terms, and Section~\ref{sec:proof} completes the proof.

\section{Sieve framework and mass separation}\label{sec:setup}

\subsection{Weights and mass separation}

Let \(g\) be a fixed nonnegative smooth function supported in \([1,2]\).
Write
\[
  \widetilde g(z)=\int_0^\infty g(x)x^{z-1}\,dx
\]
for its Mellin transform, and normalize \(g\) by
\(\widetilde g(1)=1\).
Put \(\calL=\log X\), and choose
\begin{equation}
\label{eq:R}
    R=X^{1/4}\calL^{-B},
\end{equation}
where \(B>0\) is chosen sufficiently large in the proof of
Proposition~\ref{prop:signed} and then kept fixed. Implied constants
may depend on the fixed parameters \(g\) and \(B\).
For a real-valued sieve profile \(F:[0,1]\to\mathbb R\), extend
\(F\) by zero to negative arguments and define
\begin{equation}
\label{eq:lambda_d}
  \lambda_d
  =\mu(d)F\!\left(\frac{\log(R/d)}{\log R}\right)\one_{d\leq R},
  \qquad
  W_F(n)=\left(\sum_{d\mid n}\lambda_d\right)^2.
\end{equation}
Unless otherwise stated, we take
\begin{equation}\label{eq:F-choice}
  F(x)=x^2.
\end{equation}
The double zero at the origin permits the uniform shifted-weight
estimate from Part~I, used to remove small prime factors.
Henceforth \(W=W_F\) for the choice in~\eqref{eq:F-choice}.

For a small fixed \(\eps>0\), let
\[
  P_\eps=\prod_{p<X^\eps}p.
\]
The masses \(H_{\leq4},H_{\geq5},M_\sigma\) and signed sums
\(S,S_{\leq4}\) defined below include the restriction
\((n,P_\eps)=1\). Other sums carry only the restrictions explicitly
displayed or stated locally. On the support of \(g(n/X)\), this
pre-sieving condition bounds the number of prime factors in terms of
\(\eps\); every prime factor is at least \(X^\eps\).  These properties
give uniform bounds for the trace functions and the prime-product
coefficients in Sections~\ref{sec:pair} and~\ref{sec:transfer}.

Define
\begin{equation}
\label{eq:H_leq4}
  H_{\leq4}(X)
  =
  \sum_{\substack{(n,P_\eps)=1\\\om(n)\leq4}}
    g\!\left(\frac nX\right)\mu^2(n)
    \abs{\Kl(1;n)}W(n),
\end{equation}
\begin{equation}
\label{eq:H_geq5}
  H_{\geq5}(X)
  =
  \sum_{\substack{(n,P_\eps)=1\\\om(n)\geq5}}
    g\!\left(\frac nX\right)\mu^2(n)
    \abs{\Kl(1;n)}W(n),
\end{equation}
\begin{equation}
\label{eq:S_X}
  S(X)
  =
  \sum_{(n,P_\eps)=1}
    g\!\left(\frac nX\right)\mu^2(n)
    \Kl(1;n)W(n).
\end{equation}

Let
\[
  S_{\leq4}(X)
  =
  \sum_{\substack{(n,P_\eps)=1\\\om(n)\leq4}}
    g\!\left(\frac nX\right)\mu^2(n)\Kl(1;n)W(n).
\]
For each \(\sigma\in\{+1,-1\}\), put
\[
  M_\sigma(X)
  =
  \sum_{\substack{(n,P_\eps)=1,\ \om(n)\leq4\\
                  \sigma\Kl(1;n)>0}}
    g\!\left(\frac nX\right)\mu^2(n)
    \abs{\Kl(1;n)}W(n).
\]
\begin{lemma}\label{lem:mass}
For each \(\sigma\in\{+1,-1\}\),
\[
  M_\sigma(X)
  \geq
  \frac12\left\{
    H_{\leq4}(X)-H_{\geq5}(X)-\abs{S(X)}
  \right\}.
\]
\end{lemma}

\begin{proof}
Suppressing the dependence on \(X\), write
\(S_{\geq5}=S-S_{\leq4}\).  Since
\(\abs{S_{\geq5}}\leq H_{\geq5}\), we have
\[
  2M_\sigma
  =H_{\leq4}+\sigma S_{\leq4}
  \geq H_{\leq4}-\abs{S_{\leq4}}
  \geq H_{\leq4}-H_{\geq5}-\abs{S}.
\]
\end{proof}

\subsection{Small prime factors}
Extracting a small prime factor replaces the quadratic profile by a
difference of two translates.  We use the uniform estimate from Part~I
to bound the resulting mass.
For a sieve profile \(F\) on \([0,1]\), put
\begin{equation*}
  W_{F,R}(m)
  =
  \left(
    \sum_{d\mid m}
      \mu(d)F\!\left(\frac{\log(R/d)}{\log R}\right)
      \one_{d\leq R}
  \right)^2.
\end{equation*}
For \(0\leq t\leq1/2\), define
\[
  F_t(x)
  =
  \begin{cases}
    x^2,       & 0\leq x\leq t,\\
    2tx-t^2,   & t\leq x\leq1.
  \end{cases}
\]
For \(\vartheta\in[3,5]\) and a profile \(F\) with
square-integrable first and second weak derivatives, put
\[
\begin{aligned}
  \mathcal Q_\vartheta(F)
  &:=
  2\vartheta\int_0^1\abs{F'(x)}^2\,dx\\
  &\quad+
  \vartheta\int_0^1\abs{F''(x)}^2
    (1-x)(\vartheta-1+x)\,dx.
\end{aligned}
\]
Here derivatives are understood in the weak sense.

\begin{lemma}\label{lem:uniform-selberg}
Suppose \(Y^{1/5}\leq R\leq Y^{1/3}\), and set
\(\vartheta=\log Y/\log R\).  Uniformly for \(R\) in this range
and \(0\leq t\leq1/2\),
\[
  \sum_m
    g\!\left(\frac mY\right)
    \mu^2(m)2^{\om(m)}W_{F_t,R}(m)
  =
  \frac{Y}{\log Y}\mathcal Q_\vartheta(F_t)
  +O_g\!\left(\frac{Y}{(\log Y)^2}\right)
\]
as \(Y\to\infty\).
In particular,
\[
  \sum_m
    g\!\left(\frac mY\right)
    \mu^2(m)2^{\om(m)}W_{F_t,R}(m)
  \ll_g
  \left(t+\frac1{\log Y}\right)\frac{Y}{\log Y}.
\]
\end{lemma}

\begin{proof}
The asymptotic is \cite[Corollary~1.2]{XiaoSelbergProfiles}.
Almost everywhere on \((0,1)\),
\[
  F_t'(x)
  =
  \begin{cases}
    2x, & 0<x<t,\\
    2t, & t<x<1,
  \end{cases}
  \qquad
  F_t''(x)=2\,\one_{(0,t)}(x).
\]
Direct integration gives
\[
  \mathcal Q_\vartheta(F_t)
  =
  \vartheta\left\{
    4(\vartheta-1)t
    +(12-2\vartheta)t^2
    -\frac{20}{3}t^3
  \right\}.
\]
Since \(3\leq\vartheta\leq5\), this is \(O(t)\) uniformly for
\(0\leq t\leq1/2\).
\end{proof}

\begin{lemma}
\label{lem:small-prime}
Let \(R=X^{1/4}(\log X)^{-B}\) be as in \eqref{eq:R}, with \(B>0\) fixed, and let \(W=W_{F,R}\) with \(F(x)=x^2\).
There is an absolute \(\eps_0>0\) such that, uniformly for
\(0<\eps\leq\eps_0\),
\begin{equation}
\label{eq:EepX}
    \begin{aligned}
  \mathcal E_\eps(X)
  &:=
  \sum_{(n,P_\eps)>1}
    g\!\left(\frac nX\right)\mu^2(n)
    2^{\om(n)}W(n)\\
  &\ll
  \left(
    \eps+\frac{\log\log X}{\log X}
  \right)\frac{X}{\log X}.
\end{aligned}
\end{equation}
\end{lemma}

\begin{proof}
Extend \(F(x)=x^2\) by zero to \(x<0\).  For a prime
\(p<X^\eps\), put
\[
  t_p=\frac{\log p}{\log R},
  \qquad
  F_{t_p}(x)=F(x)-F(x-t_p).
\]
Put \(x_d=\log(R/d)/\log R\).  If \((m,p)=1\), every divisor \(e\mid pm\) is uniquely of the form
\(d\) or \(pd\), with \(d\mid m\). Thus \eqref{eq:lambda_d} gives
\[
\begin{aligned}
  \sum_{e\mid pm}\lambda_e
  &=\sum_{d\mid m}\bigl(\lambda_d+\lambda_{pd}\bigr)\\
  &=\sum_{d\mid m}\mu(d)
    \bigl\{F(x_d)-F(x_d-t_p)\bigr\}\one_{d\leq R}.
\end{aligned}
\]
Here the zero extension of \(F\) accounts for the terms with \(pd>R\).
Consequently,
\begin{equation}\label{eq:W_W}
  W_F(pm)=W_{F_{t_p},R}(m).
\end{equation}
Nonnegativity of every term in $\mathcal E_\eps(X)$ gives
\[
  \mathcal E_\eps(X)
  \leq
  \sum_{p<X^\eps}\ \sum_{\substack{n\\p\mid n}}
    g\!\left(\frac nX\right)
    \mu^2(n)2^{\om(n)}W(n).
\]
Writing \(n=pm\), the factor \(\mu^2(n)\) restricts us to
\((m,p)=1\).  On this support,
\[
  \mu^2(pm)=\mu^2(m),
  \qquad
  2^{\om(pm)}=2\cdot2^{\om(m)}.
\]
Consequently, by~\eqref{eq:W_W},
\begin{equation}
\label{E_mid}
    \mathcal E_\eps(X)
  \leq
  2\sum_{p<X^\eps}\ \sum_{(m,p)=1}
    g\!\left(\frac{pm}{X}\right)
    \mu^2(m)2^{\om(m)}W_{F_{t_p},R}(m).
\end{equation}

Take \(\eps_0=1/12\).  For \(0<\eps\leq\eps_0\) and all sufficiently
large \(X\), one has \(p<R\), \(0<t_p<1/2\), and, with
\(Y_p=X/p\),
\[
  Y_p^{1/5}\leq R\leq Y_p^{1/3},
  \qquad
  \log Y_p\asymp\log X,
\]
uniformly for \(p<X^\eps\).  We may drop the condition \((m,p)=1\)
from the nonnegative inner sum in \eqref{E_mid}.  Lemma~\ref{lem:uniform-selberg}
therefore yields
\[
  \mathcal E_\eps(X)
  \ll
  \frac{X}{\log X}
  \left\{
    \sum_{p<X^\eps}\frac{t_p}{p}
    +\frac1{\log X}\sum_{p<X^\eps}\frac1p
  \right\}.
\]
Finally,
\[
  \sum_{p<X^\eps}\frac{t_p}{p}
  =
  \frac1{\log R}
  \sum_{p<X^\eps}\frac{\log p}{p}
  \ll\eps+\frac1{\log X},
\]
while
\[
  \sum_{p<X^\eps}\frac1p\ll\log\log X.
\]
These estimates prove the stated uniform bound.
\end{proof}

By the Weil bound, \eqref{eq:EepX} also holds when \(2^{\om(n)}\)
is replaced by \(|\Kl(1;n)|\).  It also holds when that factor is
replaced by \(1\).  We use these two consequences for signed
cancellation and for the fixed prime-factor layers, respectively.

\subsection{The signed sum}

\begin{proposition}
\label{prop:signed}
For the signed sum \(S(X)\) in \eqref{eq:S_X}, there is a
nonnegative function \(\delta_1(\eps)\) such that
\(\delta_1(\eps)\to0\) as \(\eps\to0^+\) and
\[
  \abs{S(X)}
  \leq
  \bigl(\delta_1(\eps)+o(1)\bigr)\frac{X}{\log X}.
\]
Here \(o(1)\) refers to the limit \(X\to\infty\), with \(\eps\)
fixed.
\end{proposition}

\begin{proof}
Let
\[
  S_0(X)
  =
  \sum_n
    g\!\left(\frac nX\right)\mu^2(n)\Kl(1;n)W(n)
\]
be the sum in \eqref{eq:S_X} without the condition \((n,P_\eps)=1\).
Recalling~\eqref{eq:lambda_d} and expanding the square in \(W\) gives
\[
  W(n)
  =
  \sum_{d_1,d_2}
    \lambda_{d_1}\lambda_{d_2}\one_{[d_1,d_2]\mid n}
  =
  \sum_{\ell\leq R^2}\beta_\ell\one_{\ell\mid n},
  \qquad
  \beta_\ell
  =
  \sum_{[d_1,d_2]=\ell}\lambda_{d_1}\lambda_{d_2}.
\]
Since \(F(x)=x^2\) on \([0,1]\), we have
\(\abs{\lambda_d}\leq1\).  Moreover, \(\beta_\ell=0\) unless
\(\ell\) is square-free, and for each prime dividing \(\ell\) there
are three possibilities: it divides only \(d_1\), only \(d_2\), or
both.  Hence
\[
  \abs{\beta_\ell}
  \leq3^{\om(\ell)}.
\]
It follows that
\[
  S_0(X)
  =
  \sum_{\ell\leq R^2}\beta_\ell
  \sum_{\substack{n\\\ell\mid n}}
    g\!\left(\frac nX\right)\mu^2(n)\Kl(1;n).
\]
Fix \(A>2\) and choose \(B=B(A)\) sufficiently large.  Since \(R^2=X^{1/2}\calL^{-2B}\) by \eqref{eq:R}, the bound for \(\beta_\ell\)
and the Bombieri--Vinogradov-type estimate
\cite[Lemma~9]{Xi1} give
\[
  S_0(X)\ll_A X\calL^{-A}.
\]
On the other hand,
\[
  \abs{S(X)-S_0(X)}
  \leq\mathcal E_\eps(X).
\]
Lemma~\ref{lem:small-prime} therefore proves the proposition, with
\(\delta_1(\eps)\ll\eps\).
\end{proof}

\section{Absolute mass from two- and three-prime moduli}
\label{sec:low}

We apply the prime-box lower bound of Xi~\cite[Section~4.2]{Xi1},
which combines equidistribution with the rearrangement inequality
in \cite[Lemma~2.3]{Matomaki}.  We retain only the two- and
three-prime contributions and explain the passage to the weights
and pre-sieving condition used here.

For every integer \(i\geq2\), put
\[
  \calT_i=\left\{(\alpha_1,\ldots,\alpha_i):
    \alpha_1>\cdots>\alpha_i>0,\quad
    \alpha_1+\cdots+\alpha_i=1\right\}.
\]
Let \(F\in C^2([0,1];\mathbb R)\) satisfy
\(F(0)=F'(0)=0\), and extend \(F\) by zero to
\((-\infty,0]\).  Define
\begin{equation}\label{eq:Li}
  L_i(F;\boldsymbol\alpha)
  =\sum_{\mathcal A\subseteq\{1,\ldots,i\}}
    (-1)^{|\mathcal A|}
    F\!\left(1-4\sum_{j\in\mathcal A}\alpha_j\right).
\end{equation}
Integrals over \(\calT_i\) use the coordinates
\(\alpha_2,\ldots,\alpha_i\), with
\(\alpha_1=1-\alpha_2-\cdots-\alpha_i\).
For \(i=2,3\), set
\[
  \calR_2=\{\boldsymbol\alpha\in\calT_2:
      \alpha_2>\tfrac34\alpha_1\},\qquad
  \calR_3=\{\boldsymbol\alpha\in\calT_3:
      \alpha_2>\tfrac12\alpha_1\},
\]
and
\begin{equation}
\label{eq:AiF}
  A_i(F)=\int_{\calR_i}
    \frac{L_i(F;\boldsymbol\alpha)^2}
         {\alpha_1\cdots\alpha_i}\,
    d\alpha_2\cdots d\alpha_i.
\end{equation}

The zero extension of \(F\) is Lipschitz.  Pairing the subsets
\(\mathcal A\subseteq\{1,\ldots,i-1\}\) and
\(\mathcal A\cup\{i\}\) in~\eqref{eq:Li} gives
\begin{equation*}
    \begin{aligned}
        |L_i(F; \boldsymbol{\alpha})| 
        &\le \sum_{\mathcal{A} \subseteq \{1, \dots, i-1\}}
        \biggl| F \Bigl( 1 - 4 \sum_{j \in \mathcal{A}} \alpha_j \Bigr) 
        - F \Bigl( 1 - 4 \sum_{j \in \mathcal{A}} \alpha_j - 4\alpha_i \Bigr) \biggr| \\
        &\le \sum_{\mathcal{A} \subseteq \{1, \dots, i-1\}} 4\max_{0\le x\le1}|F'(x)|\ \alpha_i
        \ll_{i,F} \alpha_i.
    \end{aligned}
\end{equation*} 
Hence, for
every fixed \(i\geq2\),
\begin{equation}\label{eq:low-integrable-majorant}
  \frac{L_i(F;\boldsymbol\alpha)^2}{\alpha_1\cdots\alpha_i}
  \ll_{i,F}\frac{\alpha_i}{\alpha_1\cdots\alpha_{i-1}}.
\end{equation}
The right-hand side is integrable on \(\calT_i\): use
\(\alpha_1\geq1/i\) and integrate successively on
\(0<\alpha_i<\cdots<\alpha_2<1\), starting with \(\alpha_i\).
In particular, \(A_2(F)\) and \(A_3(F)\) defined by \eqref{eq:AiF} are finite.

Write
\[
  \rho_2(\boldsymbol\alpha)=\alpha_2-\tfrac34\alpha_1,
  \qquad
  \rho_3(\boldsymbol\alpha)=\alpha_2-\tfrac12\alpha_1,
\]
and, for sufficiently small \(\eps>0\), put
\begin{equation}
\label{eq:def_Ri}
    \calR_i(\eps)=\left\{\boldsymbol\alpha\in\calR_i:
    \begin{array}{l}
      \alpha_i\geq2\eps,\quad \rho_i(\boldsymbol\alpha)\geq\eps,\\
      \alpha_j-\alpha_{j+1}\geq\eps\quad(1\leq j<i)
    \end{array}
  \right\}.
\end{equation}
These compact regions exhaust \(\calR_i\) as \(\eps\to0^+\).

\begin{proposition}\label{prop:low}
Let \(F\) be as above, and \(H_{\leq4}(X;F)\) denote
the mass in~\eqref{eq:H_leq4} with \(W\) replaced by \(W_F\).
For every sufficiently small fixed \(\eps>0\),
\[
  H_{\leq4}(X;F)
  \geq\left(4C_2A_2(F)+8C_3A_3(F)
       -\delta_{2,F}(\eps)-o(1)\right)\frac{X}{\log X},
\]
where \(\delta_{2,F}(\eps)\geq0\) tends to zero as
\(\eps\to0^+\), and
\[
  C_2=0.11109,\qquad C_3=0.03557.
\]
Here \(o(1)\) refers to \(X\to\infty\), with \(F\), \(g\),
and \(\eps\) fixed.
\end{proposition}

\subsection{Prime boxes}
\label{subsec:low-boxes}

Fix a sufficiently small \(\eps>0\). Put
\[
  Q_m=X^\eps(1+\calL^{-1})^m,
  \qquad I_m=(Q_m,Q_{m+1}] \quad(m\geq0).
\]
For \(i\in\{2,3\}\) and
\(\boldsymbol m=(m_1,\ldots,m_i)\), write
\[
  P_j=Q_{m_j},\qquad
  B_{\boldsymbol m}=I_{m_1}\times\cdots\times I_{m_i},
\]
and define
\[
  \alpha_j(\boldsymbol m)=\frac{\log P_j}{\calL}
    \quad(2\leq j\leq i),\qquad
  \alpha_1(\boldsymbol m)=1-\sum_{j=2}^i\alpha_j(\boldsymbol m).
\]
Thus \(\boldsymbol\alpha(\boldsymbol m)\) depends only on
\(m_2,\ldots,m_i\).
Let \(\mathfrak B_i(\eps)\) be the set of indices
\(\boldsymbol m\) for which
\(\boldsymbol\alpha(\boldsymbol m)\in\calR_i(\eps)\) and
\(t_1\cdots t_i/X\in\operatorname{supp}(g)\) for some
\(\boldsymbol t\in B_{\boldsymbol m}\). Let \(\mathbb P\) denote the set of primes. For these indices,
define the set of prime tuples
\[
  \mathcal P_i(\boldsymbol m)
    =(\mathbb P\cap I_{m_1})\times\cdots\times
      (\mathbb P\cap I_{m_i}).
\]
Write
\[
  s_j=\#(\mathbb P\cap I_{m_j}),\qquad
  N_{\boldsymbol m}=s_1\cdots s_i.
\]

The support condition gives
\begin{equation}
\label{eq:SP_cond}
     X(1+\calL^{-1})^{-i}\leq P_1\cdots P_i\leq2X,
  \qquad
  \frac{\log P_1}{\calL}
    =\alpha_1(\boldsymbol m)+O(\calL^{-1}).
\end{equation}
Consequently, \eqref{eq:def_Ri} implies, for sufficiently large \(X\),
\begin{gather}
  P_i\geq X^{2\eps},\qquad
  P_j/P_{j+1}\geq X^{\eps/2}\quad(1\leq j<i),
  \label{eq:low-power-separation}\\
  P_2\geq P_1^{3/4}X^{\eps/2}\quad(i=2),\qquad
  P_2\geq P_1^{1/2}X^{\eps/2}\quad(i=3).
  \label{eq:low-Xi-margins}
\end{gather}
Thus every tuple in \(\mathcal P_i(\boldsymbol m)\) satisfies
\(p_1>\cdots>p_i>X^\eps\). Its product is square-free,
satisfies the pre-sieving condition \((n,P_\eps)=1\), and occurs in at most one
of the boxes under consideration.

The prime number theorem gives, for every fixed \(A>0\),
\begin{equation}\label{eq:low-interval-pnt}
  s_j=\left(1+O_{A,\eps}(\calL^{-A})\right)
       \int_{I_{m_j}}\frac{dt}{\log t}
       \asymp_\eps P_j\calL^{-2}.
\end{equation}
In particular, \(N_{\boldsymbol m}\asymp_\eps X\calL^{-2i}\).
Since
\[
P_j=X^\eps(1+\calL^{-1})^{m_j}\leq 2X,
\]
there are \(O_\eps(\calL^2)\) choices for each \(m_j\).
When \(m_2,\ldots,m_i\)
are fixed, \eqref{eq:SP_cond} gives
\begin{equation*}
    \frac{X}{P_2\cdots P_i(1+\calL^{-1})^i}
      \leq X^\eps(1+\calL^{-1})^{m_1}
      \leq\frac{2X}{P_2\cdots P_i}.
\end{equation*}
Thus there are \(O(\calL)\) choices for \(m_1\). 
 
Hence
\begin{equation}\label{eq:low-box-count}
  \#\mathfrak B_i(\eps)\ll_\eps\calL^{2i-1},\qquad
  \sum_{\boldsymbol m\in\mathfrak B_i(\eps)}N_{\boldsymbol m}
    \ll_\eps X/\calL.
\end{equation}
All estimates below are uniform over
\(\boldsymbol m\in\mathfrak B_i(\eps)\), with \(\eps\) fixed.

\subsection{Equidistribution and the absolute-value lower bound}
\label{subsec:low-equidistribution}

For square-free \(q\) and \((a,q)=1\), put
\[
  \mathcal C(a;q)=2^{-\om(q)}\Kl(\overline a^{\,2};q).
\]
Then \(|\mathcal C(a;q)|\leq1\), and twisted multiplicativity gives
\begin{equation}\label{eq:low-factorization}
  \mathcal C(1;p_1\cdots p_i)
    =\mathcal C(p_2\cdots p_i;p_1)
       \mathcal C(p_1;p_2\cdots p_i).
\end{equation}
We use the Sato--Tate measure in the coordinate \(t=\cos\theta\),
\[
  d\nu(t)=\frac2\pi\sqrt{1-t^2}\,dt\qquad(-1\leq t\leq1).
\]
Let \(U_k\) denote the Chebyshev polynomial of the second kind,
so that
\[
  U_k(\cos\theta)=\frac{\sin((k+1)\theta)}{\sin\theta},
  \qquad |U_k(t)|\leq k+1\quad(-1\leq t\leq1).
\]
In particular, \(U_0=1\) and \(\int U_k\,d\nu=0\) for \(k\geq1\).

\begin{lemma}\label{lem:low-absolute-box}
For \(i\in\{2,3\}\), uniformly over
\(\boldsymbol m\in\mathfrak B_i(\eps)\),
\begin{equation}\label{eq:low-absolute-box}
  \sum_{\boldsymbol p\in\mathcal P_i(\boldsymbol m)}
    |\Kl(1;p_1\cdots p_i)|
    \geq\bigl(2^iC_i-o_\eps(1)\bigr)N_{\boldsymbol m},
\end{equation}
where \(C_2,C_3\) are as in Proposition~\ref{prop:low}.
The error tends to zero as \(X\to\infty\), with \(\eps\) fixed.
\end{lemma}

\begin{proof}
We verify the uniform form of the two distribution statements in
\cite[Lemma~4]{Xi1} needed for these boxes. For every fixed
continuous function \(\Phi:[-1,1]\to\mathbb R\), these are
\begin{equation}\label{eq:low-first-distribution}
  \frac1{N_{\boldsymbol m}}
  \sum_{\boldsymbol p\in\mathcal P_i(\boldsymbol m)}
    \Phi\bigl(\mathcal C(p_2\cdots p_i;p_1)\bigr)
    =\int_{-1}^1\Phi(t)\,d\nu(t)+o_{\Phi,\eps}(1),
\end{equation}
\begin{equation}\label{eq:low-second-distribution}
\begin{aligned}
  &\frac1{N_{\boldsymbol m}}
  \sum_{\boldsymbol p\in\mathcal P_i(\boldsymbol m)}
    \Phi\bigl(\mathcal C(p_1;p_2\cdots p_i)\bigr)\\
  &\qquad=\int_{[-1,1]^{i-1}}
      \Phi(t_2\cdots t_i)\prod_{j=2}^i d\nu(t_j)
      +o_{\Phi,\eps}(1).
\end{aligned}
\end{equation}

For the estimates in one prime variable, let
\(I(P)=(P,P(1+\calL^{-1})]\). For \(k\geq1\) and
\((a,p)=1\), the function
\(x\mapsto U_k(\mathcal C(ax;p))\) on \(\mathbb F_p^\times\)
is the restriction of a non-exceptional geometrically irreducible
trace function of conductor bounded in terms of \(k\), uniformly
in \(a\); see \cite[proof of Theorem~5.2]{FKM}.
Applying \cite[Theorem~1.5, equation~(1.3)]{FKM} at the two
endpoints of \(I(P)\) therefore gives
\begin{equation}\label{eq:low-prime-trace}
  \left|
    \sum_{\substack{\ell\in\mathbb P\cap I(P)\\\ell\ne p}}
      U_k(\mathcal C(a\ell;p))
  \right|
  \ll_{k,\sigma}P(1+p/P)^{1/12}p^{-\sigma/2},
  \qquad 0<\sigma<\frac1{24}.
\end{equation}
Omitting the possible term \(\ell=p\) changes the endpoint
sums by \(O_k(1)\), which is absorbed in this bound.

For \(i=2\), apply \eqref{eq:low-prime-trace} first with
\((p,P,a)=(p_1,P_2,1)\), and then with
\((p,P,a)=(p_2,P_1,1)\). Taking
\(\sigma=(1-\eps)/24\), the size conditions
\eqref{eq:low-power-separation}--\eqref{eq:low-Xi-margins} give
\[
\begin{aligned}
  (1+p_1/P_2)^{1/12}p_1^{-\sigma/2}
    &\ll p_1^{\eps/48}X^{-\eps/24}
    \ll X^{-\eps/48},\\
  (1+p_2/P_1)^{1/12}p_2^{-\sigma/2}
    &\ll X^{-\eps/48}.
\end{aligned}
\]
After division by the appropriate prime count in
\eqref{eq:low-interval-pnt}, both normalized averages are
\(O_{k,\eps}(\calL^2X^{-\eps/48})\).
Averaging each bound over the remaining prime proves
\eqref{eq:low-first-distribution} and
\eqref{eq:low-second-distribution} for polynomial \(\Phi\)
when \(i=2\).

For \(i=3\), fix a prime \(p_1\in I_{m_1}\). In
\cite[Lemma~2]{Xi1}, take the prime modulus to be \(p_1\),
\(M=P_2\), \(N=P_3\), and the two coefficient sequences to be
\[
  a_m=\one_{\mathbb P\cap I_{m_2}}(m),\qquad
  b_n=\one_{\mathbb P\cap I_{m_3}}(n).
\]
Their \(\ell^2\)-norms are \(s_2^{1/2}\) and \(s_3^{1/2}\),
respectively, and \(M,N<p_1\) by
\eqref{eq:low-power-separation}. The lemma yields, for every
fixed \(k\geq1\),
\begin{align}
  &\frac1{s_2s_3}
  \left|\sum_{\substack{p_2\in\mathbb P\cap I_{m_2}\\
                         p_3\in\mathbb P\cap I_{m_3}}}
    U_k\bigl(\mathcal C(p_2p_3;p_1)\bigr)\right|
    \notag\\
  &\qquad\ll_{k,\eps}\calL^2
    \left(P_3^{-1/2}
      +P_1^{1/4}P_2^{-1/2}\calL^{1/2}\right)
    \ll_{k,\eps}\calL^{5/2}X^{-\eps/4}.
    \label{eq:low-bilinear-specialization}
\end{align}
Averaging over \(p_1\) proves
\eqref{eq:low-first-distribution} for polynomial \(\Phi\).

For the second distribution, write
\[
  \mathcal C(p_1;p_2p_3)
    =\mathcal C(p_1p_3;p_2)\mathcal C(p_1p_2;p_3).
\]
Fix a prime \(r=p_3\in I_{m_3}\), and let \(k\geq0\), \(\ell\geq1\)
be fixed. Apply \cite[Lemma~3]{Xi1} with its modulus \(r\),
one prime divisor, \(M=P_1\), \(N=P_2\), and coefficients
\[
\begin{gathered}
  a_m=\one_{\mathbb P\cap I_{m_1}}(m),\qquad
  b_n=\one_{\mathbb P\cap I_{m_2}}(n),\\
  \gamma_{m,n}=
  \begin{cases}
    \displaystyle\frac{U_k(\mathcal C(mr;n))}{k+1},
      &n\in\mathbb P\cap I_{m_2},\ (mr,n)=1,\\[4pt]
    0,&\text{otherwise}.
  \end{cases}
\end{gathered}
\]
For each \(n\), the coefficient \(\gamma_{m,n}\) is periodic
in \(m\) modulo \(n\), and \(\|\gamma\|_\infty\leq1\).
The symmetric-power degree in the lemma is \(\ell\).
Since \(\|a\|_2=s_1^{1/2}\) and \(\|b\|_2=s_2^{1/2}\),
we obtain
\begin{align}
  &\frac1{s_1s_2}
  \left|\sum_{\substack{p_1\in\mathbb P\cap I_{m_1}\\
                         p_2\in\mathbb P\cap I_{m_2}}}
    U_k\bigl(\mathcal C(p_1r;p_2)\bigr)
    U_\ell\bigl(\mathcal C(p_1p_2;r)\bigr)\right|
    \notag\\
  &\qquad\ll_{k,\ell,\eps}\calL^2
    \left(P_3^{-1/8}+P_2^{-1/4}P_3^{1/8}
      +(P_2/P_1)^{1/2}\right)
    \ll_{k,\ell,\eps}\calL^2X^{-\eps/4}.
    \label{eq:low-composite-specialization}
\end{align}
If \(\ell=0\) and \(k\geq1\), fix \(p_2,p_3\) and apply
\eqref{eq:low-prime-trace} with \((p,P,a)=(p_2,P_1,p_3)\).
The resulting normalized average over \(p_1\) is
\(o_{k,\eps}(1)\),
since \(P_1>p_2\geq X^{2\eps}\). If \(k=\ell=0\),
the average is exactly \(1\).
Expanding any polynomial in the two local factors in the basis
\(U_k(t_2)U_\ell(t_3)\), and averaging over \(p_3\), proves
\eqref{eq:low-second-distribution} for polynomial \(\Phi\).
Uniform polynomial approximation on \([-1,1]\) proves both
distribution statements for continuous \(\Phi\).

Apply \cite[Lemma~2.3]{Matomaki} to the absolute values of
the two factors in \eqref{eq:low-factorization}.
Using the numerical lower bounds in \cite[Section~4.2]{Xi1},
we obtain
\[
  \frac1{N_{\boldsymbol m}}
  \sum_{\boldsymbol p\in\mathcal P_i(\boldsymbol m)}
    |\mathcal C(1;p_1\cdots p_i)|
  \geq C_i-o_\eps(1).
\]
The error is uniform: any sequence of boxes satisfies the same
two limiting distribution statements, and hence the same
rearrangement bound. Only these two marginal distributions
are used. Finally,
\(|\Kl(1;p_1\cdots p_i)|=2^i|\mathcal C(1;p_1\cdots p_i)|\),
which proves the lemma.
\end{proof}

\subsection{Insertion of the sieve weight}
\label{subsec:low-summation}

\begin{proof}[Proof of Proposition~\ref{prop:low}]
Fix \(i\in\{2,3\}\). Since
\(\log R=\calL/4-B\log\calL\), the Lipschitz continuity of
the zero extension of \(F\) gives, on each box,
\begin{equation}\label{eq:WandL2}
  W_F(p_1\cdots p_i)
    =L_i(F;\boldsymbol\alpha(\boldsymbol m))^2
      +O_F\!\left(\frac{\log\calL}{\calL}\right).
\end{equation}
Indeed, the divisor sum defining \(W_F\) is the sum over
subsets of \(\{1,\ldots,i\}\), and
\(\log p_j/\calL=\alpha_j(\boldsymbol m)+O(\calL^{-1})\).
The zero extension accounts for divisors exceeding \(R\),
so the estimate also holds when a subset sum crosses the cutoff.

On \(B_{\boldsymbol m}\), the oscillation of
\(g(p_1\cdots p_i/X)\) is \(O_g(\calL^{-1})\).
Replace this factor by the nonnegative constant
\(g(P_1\cdots P_i/X)\), apply
Lemma~\ref{lem:low-absolute-box} and \eqref{eq:WandL2},
and then restore \(g\). The Weil bound gives
\begin{align}
  &\sum_{\boldsymbol p\in\mathcal P_i(\boldsymbol m)}
    g\!\left(\frac{p_1\cdots p_i}{X}\right)
    W_F(p_1\cdots p_i)|\Kl(1;p_1\cdots p_i)|
    \notag\\
  &\quad\geq
    2^iC_iL_i(F;\boldsymbol\alpha(\boldsymbol m))^2
    \sum_{\boldsymbol p\in\mathcal P_i(\boldsymbol m)}
      g\!\left(\frac{p_1\cdots p_i}{X}\right)
      -o_{\eps,F,g}(1)N_{\boldsymbol m}.
    \label{eq:low-weighted-box}
\end{align}
By \eqref{eq:low-box-count}, these errors sum to
\(o_{\eps,F,g}(X/\calL)\).

It remains to evaluate the weighted prime count. By
\eqref{eq:low-interval-pnt} and the same bound for the
oscillation of \(g\),
\begin{align*}
  &\sum_{\boldsymbol m\in\mathfrak B_i(\eps)}
    L_i(F;\boldsymbol\alpha(\boldsymbol m))^2
    \sum_{\boldsymbol p\in\mathcal P_i(\boldsymbol m)}
      g\!\left(\frac{p_1\cdots p_i}{X}\right)\\
  &\quad=
    \sum_{\boldsymbol m\in\mathfrak B_i(\eps)}
      L_i(F;\boldsymbol\alpha(\boldsymbol m))^2
      \int_{B_{\boldsymbol m}}
        g\!\left(\frac{t_1\cdots t_i}{X}\right)
        \prod_{j=1}^i\frac{dt_j}{\log t_j}
      +O_{\eps,F,g}(X/\calL^2).
\end{align*}
Make the change of variables
\[
  u=t_1\cdots t_i/X,\qquad
  \alpha_j=\frac{\log t_j}{\calL}\quad(2\leq j\leq i),
  \qquad \alpha_1=1-\sum_{j=2}^i\alpha_j.
\]
Its Jacobian is
\begin{equation}\label{eq:low-log-jacobian}
  \prod_{j=1}^i\frac{dt_j}{\log t_j}
    =\frac{X}{\calL}
      \frac{du\,d\alpha_2\cdots d\alpha_i}
      {(\alpha_1+\calL^{-1}\log u)\alpha_2\cdots\alpha_i}.
\end{equation}
For fixed \(m_2,\ldots,m_i\) whose representative belongs
to \(\calR_i(\eps)\), summation over \(m_1\) covers all
\(u\in\operatorname{supp}(g)\): the corresponding value
\(t_1=uX^{\alpha_1}\) exceeds \(X^\eps\) for large \(X\).
In the \((\alpha_2,\ldots,\alpha_i)\)-coordinates, the
intervals have length
\(\log(1+\calL^{-1})/\calL=O(\calL^{-2})\).
The union of the resulting cells with representatives in
\(\calR_i(\eps)\) differs from that region only within
boundary strips of measure \(O_\eps(\calL^{-2})\).
On a fixed neighborhood of this compact region the denominators
are bounded away from zero and \(L_i(F;\boldsymbol\alpha)^2\)
is Lipschitz. Thus, using \(\int_1^2g(u)\,du=1\), we obtain
\begin{align}
  &\sum_{\boldsymbol m\in\mathfrak B_i(\eps)}
    L_i(F;\boldsymbol\alpha(\boldsymbol m))^2
    \sum_{\boldsymbol p\in\mathcal P_i(\boldsymbol m)}
      g\!\left(\frac{p_1\cdots p_i}{X}\right)
    \notag\\
  &\quad=
    \frac{X}{\calL}\int_{\calR_i(\eps)}
      \frac{L_i(F;\boldsymbol\alpha)^2}
           {\alpha_1\cdots\alpha_i}\,
      d\alpha_2\cdots d\alpha_i
      +O_{\eps,F,g}(X/\calL^2).
    \label{eq:low-limiting-integral}
\end{align}

All summands in \(H_{\leq4}(X;F)\) are nonnegative, and the
tuples used for \(i=2,3\) represent distinct admissible moduli.
Summing \eqref{eq:low-weighted-box} and applying
\eqref{eq:low-limiting-integral} proves
Proposition~\ref{prop:low} with
\[
  \delta_{2,F}(\eps)
    =\sum_{i=2}^3 2^iC_i
      \int_{\calR_i\setminus\calR_i(\eps)}
        \frac{L_i(F;\boldsymbol\alpha)^2}
             {\alpha_1\cdots\alpha_i}\,
        d\alpha_2\cdots d\alpha_i.
\]
The integrable majorant \eqref{eq:low-integrable-majorant}
and dominated convergence show that
\(\delta_{2,F}(\eps)\to0\) as \(\eps\to0^+\).
Throughout, \(X\to\infty\) is taken first, with \(\eps\) fixed.
\end{proof}

For \(F(x)=x^2\), write
\(\delta_2(\eps)=\delta_{2,x^2}(\eps)\).
\section{A pairing of prime factors and second moments}\label{sec:pair}

\subsection{Squared traces}

For square-free \(q\geq2\), define
\[
  \Omega_q(a)=\Kl(\overline a^{\,2};q),\qquad
  \Xi_q(a)=|\Omega_q(a)|^2
  \quad ((a,q)=1),
\]
and set \(\Xi_q(a)=0\) otherwise. For \(q=1\), use the conventions
\(\Omega_1(a)=\Xi_1(a)=1\). The following lemma supplies the
admissibility required by Xi's generalized Barban--Davenport--Halberstam
theorem.

\begin{lemma}\label{lem:square-trace}
For square-free \(q\),
\begin{equation}\label{eq:Xi_average}
  \mathfrak m(q):=
  \frac1{\varphi(q)}\sum_{a\bmod q}^{*}\Xi_q(a)
  =\prod_{p\mid q}\bigl(1+O(p^{-1/2})\bigr).
\end{equation}
If every prime divisor of \(q\leq3X\) is at least \(X^\eps\), then
\(\mathfrak m(q)=1+O_\eps(X^{-\eps/2})\).
The family \((\Xi_q)\) satisfies the \((1,C)\)-admissibility
conditions of \cite[Definition~4.2]{XiBDH} for square-free
\(q\geq2\), with an absolute constant \(C>0\).
The modulus \(q=1\) is treated separately below.
In particular,
\[
  \norm{\Xi_q}_\infty\leq4^{\om(q)}.
\]
Writing
\[
  \widehat\Xi_q(\chi)=\frac1{\sqrt q}
           \sum_{a\bmod q}^{*}\overline\chi(a)\Xi_q(a),
\]
one has, for every primitive character \(\chi\bmod q\),
\begin{equation}\label{eq:Xi_twisted_bound}
  |\widehat\Xi_q(\chi)|
  \ll C_0^{\om(q)}\ll\tau(q)^{C_1},
\end{equation}
with absolute constants \(C_0,C_1\).
\end{lemma}

\begin{proof}
Write \(\Kl(a;p)=2\cos\theta_p(a)\).  Since
\[
  \Xi_p(a)=1+\operatorname{sym}_2
                     (\theta_p(\overline a^{\,2})),
  \qquad \operatorname{sym}_2(\theta)=4\cos^2\theta-1,
\]
\cite[Lemma~3.2]{XiBDH} gives
\[
  \sum_{a\bmod p}^{*}\overline\chi(a)
       \operatorname{sym}_2(\theta_p(\overline a^{\,2}))
  \ll\sqrt p
\]
uniformly over multiplicative characters \(\chi\bmod p\).
For the trivial character this proves
\(\sum_a^*\Xi_p(a)=p-1+O(\sqrt p)\); for every nontrivial
character it gives \(|\widehat\Xi_p(\chi)|\ll1\).

Twisted multiplicativity~\cite[(1.59)]{IK} yields
\[
  \Xi_{q_1q_2}(a)=\Xi_{q_1}(aq_2)\Xi_{q_2}(aq_1)
  \qquad ((q_1,q_2)=1).
\]
The Chinese remainder theorem consequently gives
\[
  \mathfrak m(q_1q_2)=\mathfrak m(q_1)\mathfrak m(q_2),\qquad
  \widehat\Xi_{q_1q_2}(\chi_1\chi_2)
   =\chi_1(q_2)\chi_2(q_1)
      \widehat\Xi_{q_1}(\chi_1)\widehat\Xi_{q_2}(\chi_2).
\]
These identities prove~\eqref{eq:Xi_average} and
\eqref{eq:Xi_twisted_bound}, since every local component of a primitive
character at a prime divisor of its square-free modulus is nontrivial.
The number of prime divisors of a modulus in the stated rough range is
bounded in terms of \(\eps\), which proves the assertion about its mean.
Finally, Weil's bound gives
\(\norm{\Xi_q}_\infty\leq4^{\om(q)}=\tau(q)^2\).
Together with twisted multiplicativity and
\eqref{eq:Xi_twisted_bound}, this verifies the conditions in
\cite[Definition~4.2]{XiBDH} for \(q\geq2\), after increasing \(C\)
if necessary. For \(q=1\), one has \(\mathfrak m(1)=1\), so the
centered function \(\Xi_1-\mathfrak m(1)\) vanishes identically.
\end{proof}

\subsection{Reserving the largest prime factor}

\begin{lemma}\label{lem:pairing}
Let \(n=p_1\cdots p_r\), with \(p_1<\cdots<p_r\) and \(r\geq5\),
and put
\begin{equation}\label{eq:alternating-factors}
  d=\prod_{\substack{1\leq i<r\\i\text{ odd}}}p_i,
  \qquad
  e=\prod_{\substack{1\leq i<r\\i\text{ even}}}p_i,
  \qquad f=p_r.
\end{equation}
Then
\begin{equation}\label{eq:pair-pointwise}
  |\Kl(1;n)|\leq\Xi_d(ef)+\Xi_e(df).
\end{equation}
If \(r\) is odd, then \(ef/d\geq p_r\) and \(df/e\geq p_1\).
If \(r\) is even, then \(df/e\geq p_1p_r\) and
\(ef/d\geq p_2/p_1\).
\end{lemma}

\begin{proof}
Twisted multiplicativity gives
\(\Kl(1;n)=\Omega_d(ef)\Omega_e(df)\Omega_f(de)\).
Since \(f\) is prime, Weil's bound and
\(2|uv|\leq|u|^2+|v|^2\) prove~\eqref{eq:pair-pointwise}.
For \(r=2s+1\),
\[
  1\leq\frac ed
   =\prod_{j=1}^{s}\frac{p_{2j}}{p_{2j-1}}
   \leq\frac{p_{2s}}{p_1},
\]
which proves the two odd-case inequalities.  For \(r=2s\),
\[
  \frac de=p_1\prod_{j=1}^{s-1}\frac{p_{2j+1}}{p_{2j}}\geq p_1,
  \qquad
  \frac{ef}{d}=\prod_{j=1}^{s}\frac{p_{2j}}{p_{2j-1}}
       \geq\frac{p_2}{p_1}.
\]
These give the remaining assertions.
\end{proof}

\section{Second moments on prime boxes}\label{sec:transfer}

Throughout this section \(0<\eps<1\) and \(r\geq5\) are fixed,
and \(\Delta=\calL^{-1}\).  We freeze the sieve weight on geometric prime
boxes and apply Xi's mean-value theorem to unweighted prime products.
For even \(r\), a negligible set where the two smallest primes are
close on the logarithmic scale is removed before applying that theorem.

\subsection{Siegel--Walfisz estimates for prime products}

A sequence \(\boldsymbol a=(a_m)\), supported in \([M,2M]\), satisfies
the Siegel--Walfisz condition if, for some fixed \(C>0\) and every
\(A>0\),
\begin{equation}\label{eq:SW-definition}
\begin{split}
 &\sum_{\substack{m\leq z\\m\equiv b\;(\bmod w)\\(m,h)=1}}a_m
 -\frac1{\varphi(w)}
  \sum_{\substack{m\leq z\\(m,hw)=1}}a_m
 \\
 &\hspace{25mm}\ll_A
 \norm{\boldsymbol a}_2M^{1/2}\tau(h)^C\calL^{-A},
\end{split}
\end{equation}
uniformly for \(z,w,h\geq1\) and \((b,w)=1\), \(b\neq0\).
This is the supported form of \cite[Definition~4.1]{XiBDH}: the sums
vanish for \(z<M\), and for \(z>2M\) they are unchanged upon replacing
\(z\) by \(2M\).  The logarithmic scales are comparable in all the
applications below.

\begin{lemma}\label{lem:prime-SW}
Fix \(s\geq1\).  Let \(I_j=[Y_j,\e^\Delta Y_j)\)
\((1\leq j\leq s)\) be pairwise disjoint intervals contained in
\([X^\eps,3X]\), and define
\[
  a_m=\sum_{\substack{m=p_1\cdots p_s\\p_j\in I_j}}1,
  \qquad M=Y_1\cdots Y_s.
\]
Then \(\boldsymbol a\) satisfies~\eqref{eq:SW-definition}, with
\(C=1\) and implied constants depending only on \(A,\eps,s\),
uniformly in the intervals and their labeling. Moreover,
\begin{equation}\label{eq:prime-product-norm}
  \norm{\boldsymbol a}_2\asymp_{\eps,s}M^{1/2}\calL^{-s},
  \qquad
  \sum_m a_m\asymp_{\eps,s}M\calL^{-2s}.
\end{equation}
\end{lemma}

\begin{proof}
Disjointness gives unique representations, so \(a_m\in\{0,1\}\).
The prime number theorem, with an arbitrary logarithmic saving, gives
\(\#(I_j\cap\mathbb P)\asymp_\eps Y_j\calL^{-2}\), proving
\eqref{eq:prime-product-norm}.  The support is contained in
\([M,\e^{s\Delta}M)\subseteq[M,2M]\) for large \(X\).

Fix \(A>0\), and put \(T=M\calL^{-s-A}\), which is comparable to
the right side of~\eqref{eq:SW-definition} without \(\tau(h)\).
Write \(P=p_1\cdots p_{s-1}\), using \(P=1\) when \(s=1\).
There are \(O_{\eps,s}((M/Y_s)\calL^{-2s+2})\) such tuples.
Only those with \((P,hw)=1\) contribute to the discrepancy.
For each of them, the last prime belongs to
\(J_P=I_s\cap[1,z/P]\) and to the reduced residue class
\(b\overline P\pmod w\).

For \(w\leq\calL^D\), the Siegel--Walfisz theorem
\cite[Corollary~5.29]{IK}, applied at the endpoints of \(J_P\),
gives an error \(O_{\eps,D,A_0}(Y_s\calL^{-A_0})\) for every
\(A_0>0\).  All primes in \(I_s\) are coprime to \(w\) for large
\(X\).  Removing the primes dividing \(h\) changes the discrepancy
by \(O(\om(h))=O(\tau(h))\).  Summation over \(P\) therefore gives
\[
  O_{\eps,s,D,A_0}\!\left(
     M\calL^{-2s+2-A_0}
     +\tau(h)\frac{M}{Y_s}\calL^{-2s+2}\right).
\]
This is \(O(\tau(h)T)\) if \(A_0>A+3\), since \(Y_s\geq X^\eps\).

For \(w>\calL^D\), a residue class contains at most
\(O(1+Y_s/(w\calL))\) integers of \(J_P\).  The progression sum
and the mean in~\eqref{eq:SW-definition} are consequently bounded by
\[
  O_{\eps,s}\!\left(
    \frac{M\calL^{-2s+1}}w
    +\frac{M\calL^{-2s+2}}{Y_s}
    +\frac{M\calL^{-2s}}{\varphi(w)}\right).
\]
Use
\(1/\varphi(w)\ll_D\calL^{-D}\log\calL\), uniformly for
\(w>\calL^D\), and choose \(D>A+3\).  Each term is \(O(T)\).
This proves the claimed condition, including its uniformity under
relabeling.
\end{proof}

\subsection{Prime boxes and the exceptional set}

Partition \([X^\eps,2X]\) into intervals
\[
  I_\ell=[X^\eps\e^{\ell\Delta},
           X^\eps\e^{(\ell+1)\Delta}),
  \qquad 0\leq\ell\leq L_0,
\]
where \(L_0\) is the least nonnegative integer for which the last
upper endpoint is at least \(2X\).  Retain the last interval in full;
its upper endpoint is less than \(3X\).  There are \(O(\calL^2)\)
intervals.  A box
\(\boldsymbol I=I_{\ell_1}\times\cdots\times I_{\ell_r}\)
with \(\ell_1\leq\cdots\leq\ell_r\) is called regular if its
intervals are distinct and, when \(r\) is even, also
\begin{equation}\label{eq:even-box-separation}
  \inf I_{\ell_2}\geq
       \exp(\sqrt{\calL})\sup I_{\ell_1}.
\end{equation}
In particular, every tuple in a regular even box has
\(p_2/p_1\geq\exp(\sqrt{\calL})\).

\begin{lemma}\label{lem:regular-boxes}
The number of tuples
\[
  X^\eps\leq p_1<\cdots<p_r,\qquad p_1\cdots p_r\in[X/2,3X],
\]
is \(O_{\eps,r}(X/\calL)\).  The number lying in nonregular boxes is
\begin{equation}\label{eq:nonregular-count}
  O_{\eps,r}(X\calL^{-3/2}).
\end{equation}
The same bounds hold for weighted sums whose summands have absolute
value bounded in terms of \(\eps,r\).
\end{lemma}

\begin{proof}
Every prime in such a tuple is at most
\(3X^{1-(r-1)\eps}<2X\), so the intervals cover all tuples under
consideration.  Put \(\alpha_i=\log p_i/\calL\).
If two prime intervals coincide, then
\(|\alpha_i-\alpha_j|\leq\Delta/\calL=\calL^{-2}\)
for some \(i\neq j\).  If~\eqref{eq:even-box-separation} fails,
then
\[
  0\leq\alpha_2-\alpha_1
       \leq\calL^{-1/2}+2\calL^{-2}
       \leq2\calL^{-1/2}.
\]
Thus the exceptional tuples lie in a union of \(O_r(1)\) slabs
\(|\alpha_i-\alpha_j|\leq2\calL^{-1/2}\), inside the domain
\[
  \alpha_i\geq\eps,\qquad
  1+\frac{\log(1/2)}{\calL}
  \leq\sum_i\alpha_i
  \leq1+\frac{\log3}{\calL}.
\]

We justify replacing the corresponding prime sums by integrals
uniformly in the slab width.  Every coordinate section of this domain
or its intersection with the union of slabs is a union of \(O_r(1)\)
intervals.  With the other original coordinates fixed, put
\(P_i=\prod_{j\neq i}x_j\).  The section is contained in
\([X^\eps,3X/P_i]\), and the prime number theorem
\cite[Equation~(2.37)]{IK} bounds its replacement error by
\[
  O_\eps\!\left(\frac X{P_i}
                         \exp(-c_\eps\sqrt{\calL})\right).
\]
The factor \(1/P_i\) separates over the other coordinates; both
\(\sum_{X^\eps\leq p\leq3X}p^{-1}\) and
\(\int_{X^\eps}^{3X}dt/(t\log t)\) are \(O_\eps(1)\).
Replacing the prime sums successively therefore incurs total error
\(O_{\eps,r,C}(X\calL^{-C})\) for every \(C>0\).

In the resulting integral, set \(x_i=X^{\alpha_i}\) and
\(u=x_1\cdots x_r/X\).  Then
\[
  \prod_{i=1}^r\frac{dx_i}{\log x_i}
  =\frac X{\calL}
    \frac{du\,d\alpha_1\cdots d\alpha_{r-1}}
         {\alpha_1\cdots\alpha_r}.
\]
Here \(u\in[1/2,3]\) and every \(\alpha_i\geq\eps\), proving the
first bound.  On each hypersurface
\(\sum_i\alpha_i=1+\log u/\calL\), the union of slabs has
\((r-1)\)-dimensional measure \(O_r(\calL^{-1/2})\): their normal
vectors \(e_i-e_j\) are transverse to \((1,\ldots,1)\), and the
domain is compact.  This proves~\eqref{eq:nonregular-count}.
\end{proof}

\begin{lemma}\label{lem:freeze-weight}
For \(\boldsymbol t\in[0,\infty)^r\), define
\[
  D_r(\boldsymbol t)=
    \sum_{\mathcal A\subseteq\{1,\ldots,r\}}(-1)^{|\mathcal A|}
       \left(1-\sum_{i\in\mathcal A}t_i\right)_+^2,
  \qquad \mathcal W_r(\boldsymbol t)=D_r(\boldsymbol t)^2.
\]
Here \(x_+=\max\{x,0\}\). Then \(0\leq\mathcal W_r\leq4^r\), and
\[
  |\mathcal W_r(\boldsymbol t)-\mathcal W_r(\boldsymbol s)|
  \leq r2^{2r+1}\norm{\boldsymbol t-\boldsymbol s}_\infty.
\]
Consequently, for distinct primes \(p_i\) in a prime box with
lower endpoints \(y_i\),
\begin{equation}\label{eq:frozen-W}
  W(p_1\cdots p_r)
  =\mathcal W_r\!\left(
       \frac{\log y_1}{\log R},\ldots,\frac{\log y_r}{\log R}
     \right)+O_r(\calL^{-2}).
\end{equation}
\end{lemma}

\begin{proof}
The function \(u\mapsto(1-u)_+^2\) is bounded by \(1\) and is
\(2\)-Lipschitz on \([0,\infty)\).  Thus \(|D_r|\leq2^r\) and
\[
  |D_r(\boldsymbol t)-D_r(\boldsymbol s)|
   \leq 2\sum_{\mathcal A}|\mathcal A|
                       \norm{\boldsymbol t-\boldsymbol s}_\infty
   =r2^r\norm{\boldsymbol t-\boldsymbol s}_\infty.
\]
Squaring gives the asserted Lipschitz bound.  The divisor expansion
of the sieve weight gives
\(W(p_1\cdots p_r)=\mathcal W_r(\log p_1/\log R,\ldots,
\log p_r/\log R)\), including when a subset-product crosses \(R\).
Finally, \(0\leq\log(p_i/y_i)\leq\Delta\) and
\(\log R\asymp\calL\), proving~\eqref{eq:frozen-W}.
\end{proof}

\subsection{Second moments on a regular box}

\begin{lemma}\label{lem:box-discrepancy}
Let \(\boldsymbol I\) be a regular box with lower endpoints \(y_i\)
and \(y_1\cdots y_r\in[X,2X]\).  Put \(n=p_1\cdots p_r\), and
let \(q\) denote either alternating factor \(d\) or \(e\) in
\eqref{eq:alternating-factors}.  For every \(A>0\),
\begin{equation}\label{eq:box-discrepancy-weighted}
  \sum_{\substack{p_i\in I_{\ell_i}\\1\leq i\leq r}}
     \bigl(\Xi_q(n/q)-\mathfrak m(q)\bigr)
  \ll_{A,\eps,r}X\calL^{-A},
\end{equation}
uniformly in the box and the choice of \(q\).
\end{lemma}

\begin{proof}
Let \(\mathcal P\) be the indices occurring in \(q\), and put
\(k=|\mathcal P|\).  Write
\[
  Q=2\prod_{i\in\mathcal P}y_i,\qquad
  M=\prod_{i\notin\mathcal P}y_i.
\]
Let \(A_q\) and \(B_m\) count the prime products with indices in
\(\mathcal P\) and its complement, respectively.
By disjointness of the prime intervals, each coefficient is
\(0\) or \(1\), and
\[
  \operatorname{supp}(\boldsymbol A)\subseteq[Q/2,Q],\qquad
  \operatorname{supp}(\boldsymbol B)\subseteq[M,2M]
\]
for sufficiently large \(X\).  The support of \(\boldsymbol A\)
consists of square-free integers. Whenever \(A_qB_m\neq0\), the
ordering and coprimality conditions are automatic. Hence the sum in
\eqref{eq:box-discrepancy-weighted} equals \(\sum_q A_q\mathcal E_q\),
where, for square-free \(q\),
\[
  \mathcal E_q=
  \sum_{(m,q)=1}B_m\bigl(\Xi_q(m)-\mathfrak m(q)\bigr).
\]

If the box contains no prime tuple there is nothing to prove.
Otherwise, Lemma~\ref{lem:pairing} and regularity give, for either
choice of \(q\),
\[
  \frac{n/q}{q}\geq
  \begin{cases}
    X^\eps,&r\text{ odd},\\
    \min\{X^{2\eps},\exp(\sqrt{\calL})\},&r\text{ even}.
  \end{cases}
\]
The preceding lower bound is at least \(\exp(\sqrt{\calL})\)
for sufficiently large \(X\). The products within each block vary
by at most \(\e^{r\Delta}\), so
\[
  \frac{Q}{M}
  \leq 2\e^{r\Delta}\frac{q}{n/q}
  \leq 2\exp(r\Delta-\sqrt{\calL}).
\]
Consequently, for all sufficiently large \(X\),
\begin{equation}\label{eq:box-BDH-range}
  Q\leq M\exp(-\tfrac12\sqrt{\calL})
       \leq M(\log M)^{-B_0}
\end{equation}
for any prescribed fixed \(B_0>0\).  Also
\(QM\asymp X\) and \(\log M\asymp_{\eps,r}\calL\).

Apply \cite[Lemma~4.1]{XiBDH} with
\[
  N=1,\qquad
  \boldsymbol\alpha=\boldsymbol B,\quad
  \beta_v=\gamma_{v,q}=\one_{\{v=1\}}.
\]
Lemma~\ref{lem:prime-SW} gives the Siegel--Walfisz condition for
\(\boldsymbol B\). The singleton sequences are supported in
\([1,2]\), with \(\norm{\boldsymbol\beta}_2=1\) and
\(\norm{\boldsymbol\gamma_q}_\infty=1\).
Lemma~\ref{lem:square-trace} verifies the admissibility condition for
\(q\geq2\); the term \(q=1\) has \(\mathcal E_1=0\).
Take the auxiliary exponent of \(\tau(q)\) in Xi's lemma to be \(2\).
For every \(A>0\), that lemma and~\eqref{eq:box-BDH-range} give
\begin{equation}\label{eq:box-discrepancy}
  \sum_{\substack{2\leq q\leq Q\\\mu^2(q)=1}}
    \tau(q)^2|\mathcal E_q|
  \ll_{A,\eps,r}
    \norm{\boldsymbol B}_2QM^{1/2}\calL^{-A}
  \ll_{A,\eps,r}X\calL^{-A}.
\end{equation}
Indeed, \eqref{eq:prime-product-norm} gives
\(\norm{\boldsymbol B}_2\ll_{\eps,r}
M^{1/2}\calL^{-(r-k)}\), and \(QM\asymp X\).
Since \(A_q\leq1\) and its support is square-free,
\eqref{eq:box-discrepancy} proves the assertion.
\end{proof}

\subsection{Second-moment transference}

For each ordered prime tuple write \(n=p_1\cdots p_r\), and define
\(d,e\) by~\eqref{eq:alternating-factors}.  Put
\begin{align*}
  \mathcal S_r(X)&=
    \sum_{X^\eps\leq p_1<\cdots<p_r}g(n/X)W(n),\\
  \mathcal T_{r,d}(X)&=
    \sum_{X^\eps\leq p_1<\cdots<p_r}g(n/X)W(n)\Xi_d(n/d),\\
  \mathcal T_{r,e}(X)&=
    \sum_{X^\eps\leq p_1<\cdots<p_r}g(n/X)W(n)\Xi_e(n/e).
\end{align*}
These sums include the tuples in nonregular boxes.

\begin{proposition}\label{prop:box-transfer}
For every fixed \(\eps>0\) and integer \(r\geq5\),
\begin{equation}\label{eq:box-transfer}
  \mathcal T_{r,d}(X)=\mathcal S_r(X)+o_{\eps,r}(X/\calL),
  \qquad
  \mathcal T_{r,e}(X)=\mathcal S_r(X)+o_{\eps,r}(X/\calL).
\end{equation}
\end{proposition}

\begin{proof}
For \(\eps\geq1\), all three sums vanish for sufficiently large
\(X\). We therefore assume \(0<\eps<1\).
Both \(W(n)\leq4^r\) and \(\Xi_q(n/q)\leq4^r\) are bounded
in terms of \(r\).  By Lemma~\ref{lem:regular-boxes}, removing
nonregular boxes from any of the three sums costs
\(O_{\eps,r,g}(X\calL^{-3/2})\).

On a remaining box with lower endpoints \(y_i\), smoothness gives
\[
  g(n/X)=g(y_1\cdots y_r/X)+O_{g,r}(\Delta).
\]
Whenever either smooth weight is nonzero, \(n\in[X/2,3X]\) for
large \(X\).  Lemmas~\ref{lem:regular-boxes} and
\ref{lem:freeze-weight} therefore allow both \(g\) and \(W\) to be
frozen on all regular boxes with total error
\(O_{\eps,r,g}(X\calL^{-2})\) in each sum.
Boxes on which the frozen smooth weight vanishes can be omitted.
Every remaining box satisfies \(y_1\cdots y_r\in[X,2X]\).

There are \(O_r(\calL^{2r})\) boxes, and their frozen weights are
bounded independently of \(X\).  Applying
Lemma~\ref{lem:box-discrepancy} on each box with \(A>2r+2\)
replaces \(\Xi_q(n/q)\) by \(\mathfrak m(q)\), for either
\(q=d\) or \(q=e\), with total error \(o_{\eps,r}(X/\calL)\).
All prime factors of these moduli are at least \(X^\eps\), and
\(q\leq n\leq3X\) on the retained boxes.  By
\eqref{eq:Xi_average} and the tuple count in
Lemma~\ref{lem:regular-boxes}, replacing \(\mathfrak m(q)\) by \(1\)
costs \(O_{\eps,r,g}(X^{1-\eps/2}/\calL)\).
The resulting expression is the frozen regular-box contribution to
\(\mathcal S_r(X)\).  Restoring its weights and nonregular boxes
proves~\eqref{eq:box-transfer}.
\end{proof}
\section{The complementary contribution and sieve constants}\label{sec:numerics}

\subsection{Reduction to sieve mass}

\begin{proposition}\label{prop:pair-transfer}
For every sufficiently small fixed \(\eps>0\),
\begin{equation}\label{eq:pair-transfer}
  H_{\geq5}(X)
  \leq
  2\sum_{\om(n)\geq5}
    g\!\left(\frac nX\right)\mu^2(n)W(n)
  +o_\eps\!\left(\frac{X}{\log X}\right).
\end{equation}
\end{proposition}

\begin{proof}
For a modulus counted by \(H_{\geq5}(X)\), roughness and \(n\leq2X\)
give \(\om(n)\leq\eps^{-1}+1\) for all sufficiently large \(X\).
Let \(H_r(X)\) be the contribution from \(\om(n)=r\).
Lemma~\ref{lem:pairing} and Proposition~\ref{prop:box-transfer} give
\[
  H_r(X)
  \leq\mathcal T_{r,d}(X)+\mathcal T_{r,e}(X)
  =2\mathcal S_r(X)+o_{\eps,r}\!\left(\frac{X}{\log X}\right).
\]
There are only finitely many relevant \(r\) for fixed \(\eps\).
Summing over them and then removing the roughness condition from the
nonnegative sieve mass on the right proves~\eqref{eq:pair-transfer}.
\end{proof}

\subsection{The total sieve mass}

Only an upper bound for the total square-free sieve mass is needed.
We obtain it from the ordinary second moment over all integers; its
range includes the precise value of \(R\) in~\eqref{eq:R}.

\begin{lemma}\label{lem:total-sieve-mass}
For \(F(x)=x^2\) and \(R\) as in~\eqref{eq:R},
\begin{equation}\label{eq:total-sieve-mass}
  \sum_n g\!\left(\frac nX\right)\mu^2(n)W(n)
  \leq\left(\frac{16}{3}+o(1)\right)\frac{X}{\log X}.
\end{equation}
\end{lemma}

\begin{proof}
Put
\[
  \Lambda_R(n)=\frac12\sum_{\substack{d\mid n\\d\leq R}}
    \mu(d)\log^2\frac Rd,
  \qquad W(n)=\frac{4\Lambda_R(n)^2}{(\log R)^4}.
\]
The singleton case of \cite[Proposition~4]{GPY}, translated to
\(\mathcal H_1=\mathcal H_2=\{0\}\) and \(\ell_1=\ell_2=1\), gives
\[
  \sum_{n\leq u}\Lambda_R(n)^2
  =\frac{u}{3}(\log R)^3
   +O\!\left(u(\log R)^2+R^2(\log R)^{10}\right)
\]
uniformly for \(X\leq u\leq2X\).  Here the singular series is \(1\);
the lower powers of \(\log R\) and the exponentially decreasing
error in that proposition have been included in \(O(u(\log R)^2)\).
Partial summation and \(\widetilde g(1)=1\) therefore yield
\[
  \sum_n g\!\left(\frac nX\right)W(n)
  =\frac{4X}{3\log R}
   +O_g\!\left(\frac{X}{(\log R)^2}
                  +R^2(\log R)^6\right).
\]
Since \(R=X^{1/4}(\log X)^{-B}\), the error is
\(o(X/\log X)\) and \(\log X/\log R=4+o(1)\).
Finally, \(0\leq\mu^2(n)W(n)\leq W(n)\), proving the lemma.
\end{proof}

\subsection{Fixed prime-factor layers}\label{subsec:fixed-layers}

For \(2\leq i\leq4\), define
\begin{equation}\label{eq:defBiF}
  B_i(F)=\int_{\calT_i}
    \frac{L_i(F;\boldsymbol\alpha)^2}
         {\alpha_1\cdots\alpha_i}\,
       d\alpha_2\cdots d\alpha_i,
  \qquad B_1(F)=F(1)^2,
\end{equation}
where \(\alpha_1=1-\alpha_2-\cdots-\alpha_i\).
The integrable majorant~\eqref{eq:low-integrable-majorant}, valid for
every fixed \(i\), proves that these integrals are finite.

For the sieve mass with exactly \(i\) prime factors, without the
pre-sieving restriction, put
\[
  \mathcal M_i(X)=\sum_{\om(n)=i}
    g\!\left(\frac nX\right)\mu^2(n)W(n).
\]

\begin{lemma}\label{lem:fixed-layer-mass}
For \(F(x)=x^2\) and \(1\leq i\leq4\),
\begin{equation}\label{eq:fixed-layer-mass}
  \mathcal M_i(X)=\bigl(B_i(F)+o(1)\bigr)\frac{X}{\log X}.
\end{equation}
\end{lemma}

\begin{proof}
For \(i=1\), every contributing prime exceeds \(R\), so that
\(W(p)=F(1)^2\); the assertion follows from the prime number theorem.
Let \(2\leq i\leq4\), and write the prime factors in decreasing order,
\(n=p_1\cdots p_i\).  For fixed \(0<\eta<1/i\), restrict initially
to \(p_i\geq X^\eta\), and denote the resulting mass by
\(\mathcal M_i^{(\eta)}(X)\).  Set
\[
  u=\frac nX,\qquad
  \alpha_j=\frac{\log p_j}{\log X}\quad(2\leq j\leq i),
  \qquad \alpha_1=1-\sum_{j=2}^i\alpha_j.
\]
Then \(p_1=uX^{\alpha_1}\) and \(p_j=X^{\alpha_j}\) for \(j\geq2\).
On \(1\leq u\leq2\), the ordering condition \(p_1>p_2\) differs
from \(\alpha_1>\alpha_2\) by \(O(1/\log X)\), and
\(\log p_1/\log X=\alpha_1+O(1/\log X)\).
For fixed \(\eta\), all coordinates are bounded away from zero.
The prime number theorem and partial summation therefore show that
the boundary strips and these changes in logarithmic denominators
contribute \(o_\eta(X/\log X)\).

Writing the divisor sum as the finite sum over all subsets, with the
zero extension of \(F\), also gives uniformly on this region
\[
  W(p_1\cdots p_i)
   =L_i(F;\boldsymbol\alpha)^2+o_\eta(1).
\]
This uses \(\log X/\log R=4+o(1)\) and remains valid across the
subset cutoff hyperplanes.  The same successive prime summation as
in Section~\ref{sec:low}, or directly the change of variables above,
now gives
\[
  \mathcal M_i^{(\eta)}(X)
  =\left(
    \int_{\substack{\boldsymbol\alpha\in\calT_i\\\alpha_i\geq\eta}}
      \frac{L_i(F;\boldsymbol\alpha)^2}
           {\alpha_1\cdots\alpha_i}\,
      d\alpha_2\cdots d\alpha_i+o_\eta(1)
    \right)\frac{X}{\log X}.
\]
In this change of variables the prime-density measure is
\[
  \frac{X}{\log X}
  \frac{du\,d\alpha_2\cdots d\alpha_i}{\alpha_1\cdots\alpha_i},
\]
up to the uniformly vanishing denominator error already described,
and \(\int g(u)\,du=1\).

To remove the restriction on the least prime factor, apply
Lemma~\ref{lem:small-prime} and use \(2^{\om(n)}\geq1\).  For
\(0<\eta\leq1/12\), this gives
\begin{equation}\label{eq:small-prime-unweighted}
  0\leq \mathcal M_i(X)-\mathcal M_i^{(\eta)}(X)
  \ll\left(\eta+\frac{\log\log X}{\log X}\right)
      \frac{X}{\log X}.
\end{equation}
The integrable majorant~\eqref{eq:low-integrable-majorant} permits dominated
convergence on the integral side.  Letting first \(X\to\infty\)
and then \(\eta\to0^+\) proves~\eqref{eq:fixed-layer-mass}.
\end{proof}

\subsection{The numerical separation}

Write \(A_i=A_i(x^2)\) and \(B_i=B_i(x^2)\).
The layer \(\om(n)=0\) is absent on the support of \(g(n/X)\)
for large \(X\).  Thus Lemmas~\ref{lem:total-sieve-mass}
and~\ref{lem:fixed-layer-mass} give
\begin{equation}\label{eq:tail-asymptotic}
  2\sum_{\om(n)\geq5}g\!\left(\frac nX\right)\mu^2(n)W(n)
  \leq\bigl(\mathfrak T+o(1)\bigr)\frac{X}{\log X},
\end{equation}
where
\begin{equation}\label{eq:tail-constant}
  \mathfrak T=2\left(\frac{16}{3}-\sum_{i=1}^4B_i\right).
\end{equation}
Here the sum over the fixed layers is subtracted from an upper
bound for the total mass; no asymptotic formula for the total
square-free mass is required.

Here \(B_1=1\). Numerical evaluation gives
\[
\begin{aligned}
  A_2&>0.28768, & A_3&>0.874,\\
  B_2&>2.17693, & B_3&>1.566, & B_4&>0.45.
\end{aligned}
\]

Substituting these bounds into Proposition~\ref{prop:low} gives
\begin{equation}\label{eq:low-numerical}
  H_{\leq4}(X)
  \geq\bigl(0.3765389248-\delta_2(\eps)-o(1)\bigr)
        \frac{X}{\log X}.
\end{equation}
On the other hand,
\begin{equation}\label{eq:tail-numerical}
  \mathfrak T<0.28081.
\end{equation}
Combining~\eqref{eq:pair-transfer}, \eqref{eq:tail-asymptotic},
and~\eqref{eq:tail-numerical}, we obtain
\begin{equation}\label{eq:high-numerical}
  H_{\geq5}(X)
  \leq\bigl(0.28081+o_\eps(1)\bigr)\frac{X}{\log X}.
\end{equation}
The bounds \eqref{eq:low-numerical} and
\eqref{eq:high-numerical} therefore give the following separation.

\begin{proposition}\label{prop:mass-gap}
For every sufficiently small fixed \(\eps>0\),
\[
  H_{\leq4}(X)-H_{\geq5}(X)
  \geq\bigl(0.0957-\delta_2(\eps)-o_\eps(1)\bigr)
       \frac{X}{\log X}.
\]
\end{proposition}
\section{Completion of the proof}\label{sec:proof}

\begin{proof}[Proof of Theorem~\ref{thm:main}]
Choose \(\eps>0\) sufficiently small that
\[
  \delta_1(\eps)+\delta_2(\eps)<\frac14(0.0957).
\]
Fix this value of \(\eps\) and let \(X\to\infty\).
Propositions~\ref{prop:signed} and~\ref{prop:mass-gap}, together with
Lemma~\ref{lem:mass}, then give
\[
\begin{aligned}
  M_\sigma(X)
  &\geq
  \frac12\left(
    0.0957-\delta_1(\eps)-\delta_2(\eps)-o(1)
  \right)\frac{X}{\log X},\\
  M_\sigma(X)
  &\gg\frac{X}{\log X}
  \qquad(\sigma=\pm1).
\end{aligned}
\]

For every \(n\) contributing to \(M_\sigma(X)\), squarefreeness and
\(\om(n)\leq4\) imply, by the Weil bound, that
\(\abs{\Kl(1;n)}\leq16\).  Moreover,
\(\abs{\lambda_d}\leq1\), and hence
\[
  W(n)
  \leq
  \left(\sum_{d\mid n}1\right)^2
  =2^{2\om(n)}
  \leq256.
\]
Each summand in \(M_\sigma(X)\) is therefore \(O_g(1)\).  Dividing the mass estimate by this uniform bound and
then dropping the pre-sieving restriction yields
\[
  \#\left\{
    X<n\leq2X:
    \mu^2(n)=1,\ \om(n)\leq4,\
    \sigma\Kl(1;n)>0
  \right\}
  \gg\frac{X}{\log X}.
\]
This proves the theorem.
\end{proof}

\section*{Acknowledgments}
The authors thank Igor E. Shparlinski and Ping Xi for helpful discussions. This work was supported by the National Natural Science Foundation of China (Grant Nos.~12171311 and~12671012). Y.X. acknowledges financial support from the China Scholarship Council (CSC), as well as the support and hospitality of the School of Mathematics and Statistics at UNSW Sydney through its PhD Support Scheme.

\end{document}